\documentclass{article}
\usepackage[utf8]{inputenc}
\usepackage[T1]{fontenc}
\usepackage[english]{babel}
\usepackage{amsmath,amsthm,amssymb,amsfonts}
\usepackage{xspace}
\usepackage{hyperref}
\usepackage{verbatim}
\usepackage[hmargin=2.5cm,vmargin=3.8cm]{geometry}
\usepackage[textsize=footnotesize,color=green!40]{todonotes}

\usepackage{tikz}

\newtheorem{theorem}{Theorem}

\newtheorem{proposition}[theorem]{Proposition}
\newtheorem{corollary}[theorem]{Corollary}

\newtheorem{definition}[theorem]{Definition}

\title{Characterizing extremal graphs for open neighbourhood location-domination\footnote{\noindent The research of the first author was partially financed by the French government IDEX-ISITE initiative 16-IDEX-0001 (CAP 20-25). The research of the second author was in part supported by a grant from IPM (No. 1400050116).}}
\author{Florent Foucaud\footnote{\noindent LIMOS, CNRS UMR 6158, Universit\'e Clermont Auvergne, Aubi\`ere, France.}~\footnote{Univ. Orléans, INSA Centre Val de Loire, LIFO EA 4022, F-45067 Orléans Cedex 2, France.} \and Narges Ghareghani\footnote{\noindent Department of Industrial Design, College of  Fine Arts, University of Tehran, Tehran, Iran.}~\footnote{\noindent School of Mathematics, Institute for Research in Fundamental Sciences (IPM), P.O. Box: 19395‐5746, Tehran, Iran.} \and Aida Roshany-Tabrizi\footnote{\noindent Department of Computer Science, Augsburg University, Germany.} \and Pouyeh Sharifani\footnotemark[5]}

\begin{document}
\maketitle

\begin{abstract}
An open neighbourhood locating-dominating set is a set $S$ of vertices of a graph $G$ such that each vertex of $G$ has a neighbour in $S$, and for any two vertices $u,v$ of $G$, there is at least one vertex in $S$ that is a neighbour of exactly one of $u$ and $v$. We characterize those graphs whose only open neighbourhood locating-dominating set is the whole set of vertices. More precisely, we prove that these graphs are exactly the graphs for which all connected components are half-graphs (a half-graph is a special bipartite graph with both parts of the same size, where each part can be ordered so that the open neighbourhoods of consecutive vertices differ by exactly one vertex). This corrects a wrong characterization from the literature.
\end{abstract}

\section{Introduction}
We consider only simple, finite and undirected graphs. Many graph problems are motivated by the task of uniquely identifying the vertices using a small set of vertices. In one type of such problems, the identification is done by the neighbourhood within the solution set. More precisely, denoting $N(v)$ the open neighbourhood of a vertex $v$, an \emph{open neighbourhood locating-dominating set} (\emph{OLD set} for short) of a graph $G$ is a set $S$ of vertices of $G$ such that each vertex $v$ has a neighbour in $S$ (that is, $S$ is a \emph{total dominating set}) and for any two distinct vertices $v,w$ of $G$, $N(v)\cap S\neq N(w)\cap S$. In other words, each vertex $v$ in $G$ has a distinct and nonempty set $N(v)\cap S$ of neighbours within $S$. It is not difficult to see that a graph admits an OLD set if and only if it has no isolated vertices and contains no pair of \emph{open twins} (open twins are vertices with the same open neighbourhood). We call such graphs \emph{locatable}. The smallest size of an OLD set of a locatable graph $G$ is its open neighbourhood location-domination number, denoted by $\gamma_{OL}(G)$.

The notion of open neighbourhood location-domination was introduced in 2002 under the name of \emph{IDNT codes} by Honkala, Laihonen and Ranto in~\cite[Section 5]{IDNT} and re-discovered in 2010 by Seo and Slater in~\cite{SS10}. Various aspects of the problem were subsequently studied in~\cite{Chellali,interval-bounds,interval-algo,circulant,HY12,triangular,PP,OLDtrees}. OLD sets are related to the concept of identifying codes~\cite{KCL98} (where open neighbourhoods are replaced by closed neighbourhoods) and the earlier concept of locating-dominating sets~\cite{S87,S88} (where only vertices not in the solution set need to be uniquely identified). The variant of locating-total dominating sets is also studied~\cite{hhh06}. These concepts are special cases of more general questions of identification in discrete structures, which are studied in various settings since the 1960s~\cite{BS07,B72,MS85,R61}.

One of the first questions regarding this type of problems is, how large the optimum size of a solution (in terms of the number of vertices) can be, and what is the class of extremal examples reaching the bound. For example, this is the focus of~\cite{IDcodes-extremal} for identifying codes, and it is addressed in~\cite{Chellali} for OLD sets. In the latter paper, it is proved that there exist three connected graphs for which the OLD number is equal to the order. It is then claimed that these are the only connected graphs with this property. As we will see, this is not true. The three graphs from~\cite{Chellali} are the three smallest \emph{half-graphs}; half-graphs form an infinite family of bipartite graphs, studied and named by Erd\H{o}s and Hajnal (see~\cite{EH84}). See Figure~\ref{fig:half-graphs} for the first five half-graphs. Half-graphs, or rather their complements, have been useful in previous works about locating-dominating sets~\cite{Heia} and identifying codes~\cite{IDcodes-extremal}.

We show that each half-graph $H$ of order $n$ satisfies $\gamma_{OL}(H)=n$. Note that if a locatable graph is disconnected, its OLD number is the sum of those of its connected components, thus one obtains other examples by taking disjoint unions of half-graphs. Our main theorem proves that these are the only examples.

\begin{theorem}\label{thm:main}
For a connected locatable graph $G$ of order $n$, $\gamma_{OL}(G)=n$ if and only if $G$ is a half-graph. 
\end{theorem}

This characterization corrects the incomplete one from~\cite{Chellali}. It also reveals an interesting property of OLD sets: indeed, there are only few concepts known to us in the area of graph identification problems, for which there exist infinitely many finite connected undirected graphs which have their whole vertex set as unique solution. (Another concept with this feature is the one of self-locating dominating sets, see~\cite{solid}.)

\section{The characterization}

In a locatable graph $G$, some vertices have to belong to any OLD set: we call such vertices \emph{forced}. There are two types of forced vertices: those that are forced because of the domination condition, and those that are forced because of the location condition.

\begin{definition} Let $G$ be a locatable graph and let $v$ be a vertex of $G$.
Vertex $v$ is called \emph{domination-forced} if there exists a vertex $w$, such that $v$ is the unique neighbour of $w$. Vertex $v$ is called \emph{location-forced} if there exist two distinct vertices $x$ and $y$, such that $N(x)\ominus N(y)=\{v\}$ (where $\ominus$ denotes the set symmetric difference).
\end{definition}

We can observe the following.

\begin{proposition}\label{prop:all-forced}
 If there is a vertex $v$ in a locatable graph $G$ which is neither domination-forced nor location-forced, then $V(G)\setminus\{v\}$ is an OLD set of $G$.
\end{proposition}
\begin{proof}
Since $v$ is not domination-forced, every vertex of $G$ has a neighbour in $V(G)\setminus\{v\}$. Moreover since $v$ is not location-forced, for every pair $z,w$ of distinct vertices in $G$, there is a vertex in $V(G)\setminus\{v\}$ in the symmetric difference $N(z)\ominus N(w)$, which therefore distinguishes $z$ and $w$.
\end{proof}

Proposition~\ref{prop:all-forced} implies that in any locatable graph $G$ of order $n$ with $\gamma_{OL}(G)=n$, every vertex is domination-forced or location-forced (or both).

We now formally define half-graphs, which were named in~\cite{EH84}.

\begin{definition}
For any integer $k\geq 1$, the half-graph $H_k$ is the bipartite graph on vertex sets $\{v_1,\ldots,v_k\}$ and $\{w_1,\ldots,w_k\}$, with an edge between $v_i$ and $w_j$ if and only if $i\leq j$.
\end{definition}

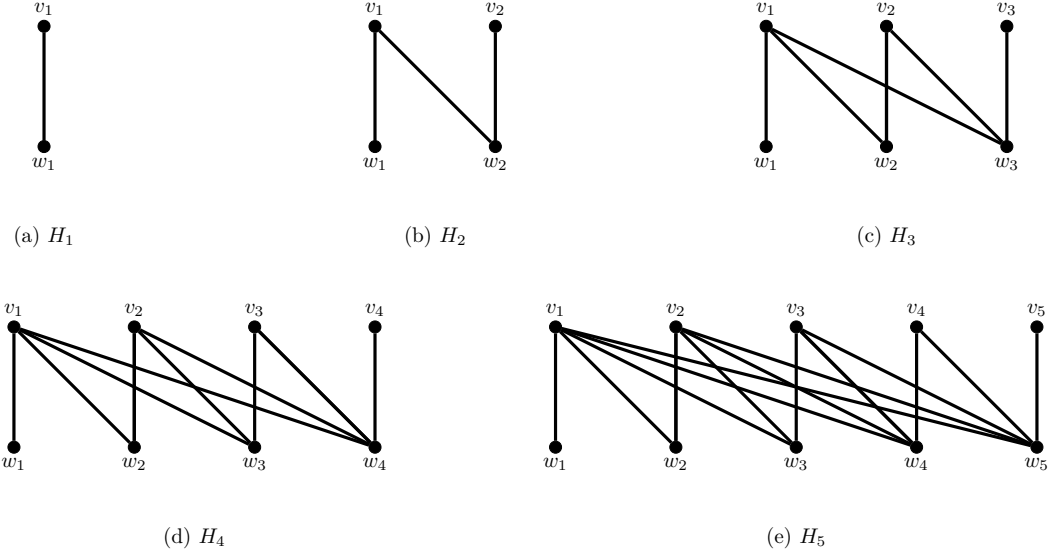
\begin{figure}[!htpb]
\centering
\scalebox{0.8}{
\begin{tikzpicture}[inner sep=2pt, minimum size=2mm]

\begin{scope}

  \node[circle, draw=black!100,fill=black!100,label={270:$w_1$}] (w1) at (0,0) {};
  \node[circle, draw=black!100,fill=black!100,label={90:$v_1$}] (v1) at (0,2) {};

  \draw[line width = 1.5pt] (w1) -- (v1);
  
    \node (title) at (0,-3/2) {(a) $H_1$};
  
\end{scope}

\begin{scope}[xshift=7.5cm]
  \node[circle, draw=black!100,fill=black!100,label={270:$w_1$}] (w1) at (-2,0) {};
  \node[circle, draw=black!100,fill=black!100,label={90:$v_1$}] (v1) at (-2,2) {};
  \node[circle, draw=black!100,fill=black!100,label={270:$w_2$}] (w2) at (0,0) {};
  \node[circle, draw=black!100,fill=black!100,label={90:$v_2$}] (v2) at (0,2) {};
  
  \draw[line width = 1.5pt] (w1) -- (v1);
  \draw[line width = 1.5pt] (v1) -- (w2);
  \draw[line width = 1.5pt] (w2) -- (v2);

    \node (title) at (-1,-3/2) {(b) $H_2$};
\end{scope}

\begin{scope}[xshift=14cm]
\node[circle, draw=black!100,fill=black!100,label={270:$w_1$}] (w1) at (-2,0) {};
  \node[circle, draw=black!100,fill=black!100,label={90:$v_1$}] (v1) at (-2,2) {};
  \node[circle, draw=black!100,fill=black!100,label={270:$w_2$}] (w2) at (0,0) {};
  \node[circle, draw=black!100,fill=black!100,label={90:$v_2$}] (v2) at (0,2) {};
  \node[circle, draw=black!100,fill=black!100,label={90:$v_3$}] (v3) at (2,2) {};
  \node[circle, draw=black!100,fill=black!100,label={270:$w_3$}] (w3) at (2,0) {};
  
  \draw[line width = 1.5pt] (w1) -- (v1);
  \draw[line width = 1.5pt] (v1) -- (w2);
  \draw[line width = 1.5pt] (w2) -- (v2);
  \draw[line width = 1.5pt] (v1) -- (w3);
  \draw[line width = 1.5pt] (v2) -- (w2);
  \draw[line width = 1.5pt] (w3) -- (v3);
  \draw[line width = 1.5pt] (v2) -- (w3);

    \node (title) at (0,-3/2) {(c) $H_3$};
\end{scope}

\begin{scope}[xshift=3.5cm,yshift=-5cm]
  \node[circle, draw=black!100,fill=black!100,label={270:$w_1$}] (w1) at (-4,0) {};
  \node[circle, draw=black!100,fill=black!100,label={90:$v_1$}] (v1) at (-4,2) {};
  \node[circle, draw=black!100,fill=black!100,label={270:$w_2$}] (w2) at (-2,0) {};
  \node[circle, draw=black!100,fill=black!100,label={90:$v_2$}] (v2) at (-2,2) {};
  \node[circle, draw=black!100,fill=black!100,label={90:$v_3$}] (v3) at (0,2) {};
  \node[circle, draw=black!100,fill=black!100,label={270:$w_3$}] (w3) at (0,0) {};
  \node[circle, draw=black!100,fill=black!100,label={90:$v_4$}] (v4) at (2,2) {};
  \node[circle, draw=black!100,fill=black!100,label={270:$w_4$}] (w4) at (2,0) {};
  
  \draw[line width = 1.5pt] (w1) -- (v1);
  \draw[line width = 1.5pt] (v1) -- (w2);
  \draw[line width = 1.5pt] (w2) -- (v2);
  \draw[line width = 1.5pt] (v1) -- (w3);
  \draw[line width = 1.5pt] (v2) -- (w2);
  \draw[line width = 1.5pt] (w3) -- (v3);
  \draw[line width = 1.5pt] (v2) -- (w3);
  \draw[line width = 1.5pt] (v3) -- (w4);
  \draw[line width = 1.5pt] (v1) -- (w4);
  \draw[line width = 1.5pt] (v2) -- (w4);
  \draw[line width = 1.5pt] (v3) -- (w4);
  \draw[line width = 1.5pt] (v4) -- (w4);

    \node (title) at (-1,-3/2) {(d) $H_4$};
\end{scope}

\begin{scope}[xshift=10.5cm,yshift=-5cm]
  \node[circle, draw=black!100,fill=black!100,label={270:$w_1$}] (w1) at (-2,0) {};
  \node[circle, draw=black!100,fill=black!100,label={90:$v_1$}] (v1) at (-2,2) {};
  \node[circle, draw=black!100,fill=black!100,label={270:$w_2$}] (w2) at (0,0) {};
  \node[circle, draw=black!100,fill=black!100,label={90:$v_2$}] (v2) at (0,2) {};
  \node[circle, draw=black!100,fill=black!100,label={90:$v_3$}] (v3) at (2,2) {};
  \node[circle, draw=black!100,fill=black!100,label={270:$w_3$}] (w3) at (2,0) {};
  \node[circle, draw=black!100,fill=black!100,label={90:$v_4$}] (v4) at (4,2) {};
  \node[circle, draw=black!100,fill=black!100,label={270:$w_4$}] (w4) at (4,0) {};
  \node[circle, draw=black!100,fill=black!100,label={90:$v_5$}] (v5) at (6,2) {};
  \node[circle, draw=black!100,fill=black!100,label={270:$w_5$}] (w5) at (6,0) {};
  
  \draw[line width = 1.5pt] (w1) -- (v1);
  \draw[line width = 1.5pt] (v1) -- (w2);
  \draw[line width = 1.5pt] (w2) -- (v2);
  \draw[line width = 1.5pt] (v1) -- (w3);
  \draw[line width = 1.5pt] (v2) -- (w2);
  \draw[line width = 1.5pt] (w3) -- (v3);
  \draw[line width = 1.5pt] (v2) -- (w3);
  \draw[line width = 1.5pt] (v3) -- (w4);
  \draw[line width = 1.5pt] (v1) -- (w4);
  \draw[line width = 1.5pt] (v2) -- (w4);
  \draw[line width = 1.5pt] (v3) -- (w4);
  \draw[line width = 1.5pt] (v4) -- (w4);
  \draw[line width = 1.5pt] (v1) -- (w5);
  \draw[line width = 1.5pt] (v2) -- (w5);
  \draw[line width = 1.5pt] (v3) -- (w5);
  \draw[line width = 1.5pt] (v4) -- (w5);
  \draw[line width = 1.5pt] (v5) -- (w5);

    \node (title) at (2,-3/2) {(e) $H_5$};
\end{scope}

\end{tikzpicture}}
\caption{The five first half-graphs.}\label{fig:half-graphs}
\end{figure}

We now show that every half-graph $H_k$ of order $n=2k$ satisfies $\gamma_{OL}(H_k)=n$.

\begin{proposition}\label{prop:half-graphs}
For every $k\geq 1$, the half-graph $H_k$ satisfies $\gamma_{OL}(H_k)=2k$.
\end{proposition}
\begin{proof}
Let $k\geq 1$. It is clear that $H_k$ is locatable, thus $\gamma_{OL}(H_k)\leq 2k$. By Proposition~\ref{prop:all-forced}, it suffices to prove that every vertex of $H_k$ is forced.
Vertices $v_1$ and $w_k$ are domination-forced, since they are the only neighbours of $w_1$ and $v_k$, respectively.
For every integer $i$ with $1\leq i\leq {k-1}$, $w_i$ is the only vertex in the symmetric difference of $N(v_i)$ and $N(v_{i+1})$. Thus, $w_i$ is location-forced.
Since graph $H_k$ is symmetric, similarly $v_k,\ldots,v_2$ are location-forced. This completes the proof.
\end{proof}

Before proving our characterization, we will use the following celebrated theorem of Bondy to upper-bound the number of locating-forced vertices in any locatable graph.

\begin{theorem}[Bondy's theorem \cite{B72}]
Let $V$ be an $n$-set, and $\mathcal{A}=\{\mathcal{A}_1,\mathcal{A}_2,\ldots,\mathcal{A}_n\}$ be a family of $n$ distinct subsets of $V$. There is a $(n-1)$-subset $X$ of $V$ such that the sets $\mathcal{A}_1\cap X, \mathcal{A}_2\cap X, \mathcal{A}_3\cap X,\ldots, \mathcal{A}_n\cap X$ are still distinct.
\end{theorem}

\begin{corollary}\label{corollary 1.4.1}
Every locatable graph $G$ of order $n$ has at most $n-1$ location-forced vertices.
\end{corollary}
\begin{proof}
Construct from graph $G$ the set system with $V(G)$ as its $n$-set and where the $A_i$'s are all the open neighbourhoods of vertices of $G$. Bondy's theorem implies that there is one vertex such that removing it does not create two same open neighbourhoods. In other words, this vertex is not location-forced.
\end{proof}

We are now ready to prove Theorem~\ref{thm:main} that a connected locatable graph $G$ of order $n$ satisfies $\gamma_{OL}(G)=n$ if and only if $G$ is a half-graph. In the proof, we will say that a vertex $v$ of a graph $G$ is \emph{located} by a set $S$ if there is no vertex $w\neq v$ with $N(v)\cap S=N(w)\cap S$. Also, $v$ is \emph{total dominated} by $S$ if $v$ has a neighbour in $S$.

\begin{proof}[Proof of Theorem~\ref{thm:main}]
The sufficient side is proved in Proposition \ref{prop:half-graphs}. We prove the necessary side by induction on $n$. Let $G$ be a connected locatable graph of order $n$ with $\gamma_{OL}(G)=n$. We first prove the statement for $n\leq 3$. The graph of order~$1$ is not locatable. The only locatable connected graph of order~$2$ is $K_2$ (which is also the half-graph $H_1$), and $\gamma_{OL}(K_2)=2$. The only locatable connected graph of order~$3$ is $K_3$, and $\gamma_{OL}(K_3)=2$.

Suppose $G$ is a connected locatable graph of order $n$ such that $n\geq 4$ and $\gamma_{OL}(G)=n$. By Proposition~\ref{prop:all-forced}, every vertex in $G$ is either domination-forced or location-forced. By Corollary~\ref{corollary 1.4.1}, there is at least one vertex that is domination-forced, let us call it $x$. Let $y$ be a vertex such that $x$ is the unique neighbour of $y$. We claim that vertex $y$ is location-forced. If $y$ was domination-forced, then there would exist a vertex that is dominated only by $y$. That vertex should be equal to $x$, and then $G$ would be $K_2$, a contradiction. Therefore, $y$ is location-forced and so, there exists a vertex $z$, such that $N(x)=N(z)\cup\{y\}$.

Now we remove $x$ and $y$ from $G$ and call the new graph $G'$. We claim that $G'$ is locatable and connected. Since every vertex (other than $y$) that is a neighbour of $x$ is a neighbour of $z$, and $y$ has no neighbours in $V(G')$, $G'$ remains connected. This implies that $G'$ has no isolated vertices (since $n\geq 4$). To see that $G'$ is locatable, assume by contradiction that there are two vertices $s$ and $t$ that are open twins in $G'$. Thus, in $G$, without loss of generality, $N(s)=N(t)\cup\{x\}$. But recall that $N(x)=N(z)\cup\{y\}$. This means that $z$ is a neighbour of $s$, and not a neighbour of $t$. So $s$ and $t$ are not open twins in $G'$, a contradiction. Thus, $G'$ is locatable.

We now claim that $\gamma_{OL}(G')=n-2$. By contradiction, suppose $\gamma_{OL}(G')\leq n-3$. Let $S'$ be an OLD set of $G'$, with $|S'|\leq n-3$. We claim that $S=S'\cup \{x,y\}$ is an OLD set of $G$. First of all, $S$ is total dominating, indeed, every vertex in $G'$ has a neighbour in $S'$, and $x$ and $y$ are total dominated by each other. Secondly, we have to check that $S$ locates all the vertices. Vertex $x$ is located by $S$, since it is the only vertex that has $y$ as its neighbour.
In addition, we claim that $y$ is located by $S$. Indeed, if it was not, there would be a vertex $w$ such that $N(w)\cap S=\{x\}$. Thus $N(w)\cap S'=\emptyset$. Therefore, $w$ is not total dominated by $S'$ in $G'$, a contradiction. So $y$ is located by $S$. In addition, all the vertices of $G'$ are still located because they were located in $G'$ by $S'$, and they remain located in $G$ by the vertices in $S'$. As a result, $S$ is an OLD set of $G$. Thus, if $|S'|\leq n-3$, then $|S|\leq n-1$, which is a contradiction. Therefore, $\gamma_{OL}(G')=n-2$.
By induction, we conclude that $G'$ is isomorphic to the half-graph $H_k$ with $k=\frac{n}{2}-1$ and vertex set $\{v_1,v_2,...,v_k\}\cup \{w_1,w_2,...,w_k\}$.

Now we claim that vertex $z$ is domination-forced in $G'$. 
Note that vertex $z$ is location-forced in $G$ (indeed $z$ cannot be domination-forced in $G$, since all its neighbours are also neighbours of $x$), which means that there exist vertices $v$ and $w$ in $G$ such that $N(w)=N(v)\cup\{z\}$. Vertex $v$ cannot be in the common neighbourhood of $x$ and $z$, and since $N(x)=N(z)\cup\{y\}$, we have $v=y$. Hence, $w$ has degree~$2$ in $G$ and $N(w)=\{x,z\}$. Therefore, by removing $x$ and $y$ from $G$ to obtain $G'$, $z$ is the only neighbour of $w$ in $G'$. 
Thus, $z$ is domination-forced in $G'$, as claimed.

As we know, in $H_k$ there are two domination-forced vertices: $v_1$ and $w_k$. By the symmetry of $H_k$, without loss of generality, we can assume that $z=w_k$. As a result, it follows from the fact that $N(x)=N(z)\cup\{y\}$, that we can label $x=w_{k+1}$ and $y=v_{k+1}$ and that $G$ is isomorphic to the half-graph $H_{k+1}$. This completes the proof.
\end{proof}

\end{document}